\documentclass[twoside]{article}

\usepackage{graphics}
\usepackage{amsmath,amsopn,amssymb,epsfig,psfrag}
\usepackage{color,xcolor}
\usepackage{url}
\usepackage{graphicx}
\usepackage{subfig}
\usepackage{float}
\usepackage{amsfonts}
\usepackage{multirow}
\usepackage{bm}
\usepackage{makecell}
\usepackage{xr}
\usepackage[version=4]{mhchem}

\usepackage[]{mathtools}

\DeclareMathOperator{\bigO}{\mathcal{O}}
\usepackage[]{tikz-cd}
\usepackage[]{dsfont}

\usepackage[colorlinks]{hyperref}

\usepackage{amsthm}

\newtheorem{theorem}{Theorem}[section]
\newtheorem{lemma}{Lemma}[section]

\theoremstyle{definition}

\newtheorem{remark}{Remark}[section]
\newtheorem{example}{Example}[section]

\DeclareMathOperator{\rank}{rank}
\DeclareMathOperator{\diag}{diag}

\DeclareMathOperator{\HH}{\sf H}
\DeclareMathOperator{\T}{\sf T}

\def\bbC{\mathbb{C}}

\makeatletter
\tagsleft@false\let\veqno\@@eqno
\makeatother

\def\ol{\overline}

\title{A decoupled form of the structure-preserving \\ doubling algorithm with low-rank structures}

\date{}

\author{
Zhen-Chen Guo\thanks{Department of Mathematics, Nanjing University,  Nanjing  210093, People's Republic of China; \texttt{e-mail: guozhenchen@nju.edu.cn}} \and 
Eric King-wah Chu\thanks{School of Mathematics, 9 Rainforest Walk, Monash University, Victoria 3800, Australia; {\tt  e-mail: eric.chu@monash.edu}} \and
Xin Liang\thanks{Yau Mathematical Sciences Center, Tsinghua University, Beijing 10084, People's Republic of China; \texttt{e-mail: liangxinslm@tsinghua.edu.cn}} \and 
Wen-Wei Lin\thanks{Department of Applied Mathematics, National Chiao Tung University, Hsinchu 300, Taiwan; \texttt{e-mail: wwlin@math.nctu.edu.tw}}.
}
\begin{document}
%
%
%
\newcommand{\bs}{\begin{slide}{}}
\newcommand{\es}{\end{slide}} \newcommand{\be}{\begin{eqnarray*}}
\newcommand{\ee}{\end{eqnarray*}}
\newcommand{\bt}{\begin{tabular}}
\newcommand{\et}{\end{tabular}} \newcounter{bean}
\newcommand{\bl}[1]{\begin{list}{#1}{\usecounter{bean}}} \newcommand{\el}{\end{list}}
\newcommand{\bel}[1]{\begin{equation} \label{#1}} \newcommand{\eel}{\end{equation}}
\newenvironment{mat}{\left[\begin{array}}{\end{array}\right]}
\newenvironment{dt}{\left|\begin{array}}{\end{array}\right|} \newcommand{\bip}{\langle}
\newcommand{\eip}{\rangle} \newenvironment{arr}{\begin{array}}{\end{array}} \def\xa{x_{1}}
\def\xb{x_{2}} \def\xc{x_{3}} \newcommand{\ebs}{\end{slide}\begin{slide}{}}
\newcommand{\bc}{\begin{center}} \newcommand{\ec}{\end{center}}
\newcommand{\Lra}{\Leftrightarrow} \newcommand{\lra}{\leftrightarrow}
\newcommand{\Llra}{\Longleftrightarrow} \newcommand{\La}{\Leftarrow}
\newcommand{\Ra}{\Rightarrow} \newcommand{\la}{\leftarrow} \newcommand{\ra}{\rightarrow}
\def\eop{\hfill $\Box$ \vspace{5mm}}
\def\diag{\mathrm{diag}} \def\rank{\mathrm{rank}}
\def\sr{{\cal R}} \def\cs{{\cal C}} \def\vx{{\bf x}}
\def\rank{\mathrm{r}}
\maketitle

\begin{abstract}
	The structure-preserving doubling algorithm (SDA)  is a fairly efficient method for solving problems closely related to Hamiltonian (or Hamiltonian-like) matrices, such as  computing the required solutions to  algebraic Riccati equations. However, for  large-scale problems in $\mathbb{C}^n$ (also $\mathbb{R}^n$), the SDA with an $\bigO(n^3)$ computational complexity does not work well. In this paper, we propose a new  decoupled form of the  SDA (we name it as dSDA), building on the associated Krylov subspaces thus leading to the inherent low-rank structures. Importantly, the approach decouples the original two to four iteration formulae. 	The resulting dSDA is much more efficient since only one quantity (instead of the original two to four) is computed iteratively. For large-scale problems, further efficiency is gained from the low-rank structures. This paper presents the theoretical aspects of the dSDA. A practical algorithm dSDA$\rm{_t}$ with truncation and many  illustrative numerical results will appear in a second paper. 
\end{abstract}


\pagestyle{myheadings} \thispagestyle{plain}
\markboth{Z.-C. Guo,  E.K.-W. Chu, X. Liang and W.-W. Lin}
{DSDA: a Decoupled Form for the SDA}

\section{Introduction} 

The  doubling algorithm (DA), in some sense, skips many items in the iteration process and only computes the $k$-th iterates with  $k=2^j, j=0, 1, 2, \cdots$.  
The  DA idea can  at least be traced back to, to the best of our knowledge, the nineteen seventies ---  in \cite{davisonM1973numerical, friedlanderKL1976scatteringII,lainiotis1976partitioned} the DA was adopted to solve the  matrix Riccati differential equations. 
In 1978 Anderson~\cite{anderson1978secondorder}   compendiously surveyed  the existing DAs at that time and firstly  introduced the structure-preserving doubling algorithm (SDA) for algebraic Riccati equations. 
In the last two decades or more, an enormous amount of research efforts goes into the remarkable method, including theories, numerical algorithms and efficient implementation;  
please  consult~\cite{chuFLW2004structurepreserving, huangLL2018structurepreserving, biniIM2012numerical, chuFL2005structurepreserving} and the references therein.  
In~\cite{chuFLW2004structurepreserving}, Chu et al. revisited the  SDA and applied  successively  it to the  periodic discrete-time algebraic Riccati equations. 
Since then, the SDA has been generalized for many  matrix equations, such as  the  continuous-time algebraic equations~\cite{chuYL2005structurepreserving, huangL2009structured,huangLL2018structurepreserving}, 
the  M-matrix algebraic Riccati equations~\cite{guoLX2006structurepreserving, xueL2017highly,xueXL2012accurate, guoIM2007doubling}, and the $H^{*}-$matrix algebraic Riccati equations~\cite{liuX2016complex}.  
Some related eigenvalue problems, such as the palindromic quadratic eigenvalue problems~\cite{guoL2010solving, luYL2015new,luWKLL2016fast} and the Bethe-Salpeter eigenvalue problems~\cite{guoCL2019doubling}, have also been treated. 

The classical SDA possesses an $\bigO(n^3)$ computational complexity for problems in $\mathbb{R}^n$ or $\mathbb{C}^n$, and is best suited for moderate values of  $n$, with its global and quadratic convergence~\cite{linX2006convergence}; for the linear convergence in the critical case, please consult \cite{huangL2009structured}. 
However, for large-scale problems, the original SDA obviously does not work efficiently, because of its computational complexity, or the high costs in memory requirement and execution time.

In this paper  we emphasize on the numerical solution of large-scale algebraic Riccati equations (AREs) with low-rank structures  by the SDA. We consider the discrete-time algebraic Riccati equations (DAREs), the continuous-time algebraic Riccati equations (CAREs) and the M-matrix algebraic Riccati equations (MARE).   
For these AREs,  the SDA has three or four coupled recursions (see \eqref{eq:sda} and \eqref{eq:maresda} in Section~\ref{sec:structure-preserving-doubling-algorithm}). For large-scale problems however, one recursion has been applied implicitly (because of the loss of sparsity), leading to a flop count with an exponentially increasing constant thus inefficiency. We  propose a {\it new\/} form of the SDA (namely dSDA),  which decouples the three (see~\eqref{eq:sda}) or  four (see~\eqref{eq:maresda})  recursions. 
Because of the decoupling, the dSDA computes more efficiently, with only one recursion for the desired numerical solution.   

Our main contributions are summarized as follows:
\begin{enumerate}
	\item We decouple the three recursions \eqref{eq:sda} and four recursions \eqref{eq:maresda}  in the original SDA and develop the dSDA. On the surface, the new method is closely related to the Krylov subspace projection methods but the dSDA inherits the sound theoretical foundation of the SDA. 
    	\item  The original SDA has three (or four) iteration formulae \eqref{eq:sda} (or \eqref{eq:maresda}) for $A_k$ (or $E_k$ and $F_k$), $G_k$ and $H_k$ in  the $k$-th iteration. The dSDA for large-scale AREs no longer requires $A_k$ (or $E_k$ and $F_k$), thus eliminating the $2^k$ factor in the flop count and improving the efficiency tremendously. We only compute $H_k$, the desired approximate solution of the ARE.
\end{enumerate}

This paper presents the theoretical aspects of the dSDA. A practical algorithm dSDA$\rm{_t}$ with truncation and the illustrative numerical results will appear in a second paper.

\subparagraph{Notations}\label{spara:notations}
The null matrix is $0$ and the   $n$-by-$n$ identity matrix is denoted by $I_n$, with its subscript ignored when the size is clear;  $(\cdot)^{\HH}$ and $(\cdot)^{\T}$  take the conjugate transpose and the transpose of matrices, respectively. The complex conjugate of a matrix $A$ is  $\overline{A}$. The 2-norm is denoted by $\| \cdot\|$. By $M\oplus N$, we denote $\begin{bmatrix}
	M&0\\0&N
\end{bmatrix}$.

\subparagraph{Organization}\label{spara:organization}
We  revisit the SDA  for the DAREs, CAREs and MAREs in Section~\ref{sec:structure-preserving-doubling-algorithm},  and  then develop the dSDA  for these  AREs  in Section~\ref{sec:decoupled-form-for-sda}; the SDA is also extended for the Bethe-Salpeter eigenvalue problems (BSEPs).  Some numerical results are presented in Section~\ref{sec:numerical-example}, and  
some conclusions are drawn in Section~\ref{sec:conclusions}.

The Sherman-Morrison-Woodbury formula (SMWF): 
\begin{equation}\label{eq:smwf}
(M + UDV^{\T})^{-1} = M^{-1} - M^{-1} U (D^{-1} + V^{\T} M^{-1} U)^{-1} V^{\T} M^{-1},
\end{equation}
with the inverse sign indicating invertibility, will be applied occasionally.

\section{Structure-preserving doubling algorithm}\label{sec:structure-preserving-doubling-algorithm}

Consider the linear time-invariant control system in continuous-time:
\[
\dot{x}(t) = Ax(t) + Bu(t), \ \ \ y(t) = Cx(t),
\]
where $A \in \mathbb{R}^{n \times n}$, $B \in \mathbb{R}^{n \times m}$ and $C \in \mathbb{R}^{l \times n}$ with $m,l \leq n$, $x(t)$ is the state vector and $u(t)$ is the control vector.  
The linear-quadratic (LQ)  optimal control minimizes the cost functional
$J_c(x,u) \equiv \int_0^\infty \left[ x(t)^{\T} H x(t) + u(t)^{\T} R u(t) \right] dt$,
with $H \equiv C^{\T} C \geq 0$ and $R > 0$. Here, a symmetric matrix $M > 0$ ($\geq 0$) when all its eigenvalues are positive (non-negative). Also, $M > N$ ($M \geq N$) if and only if $M - N > 0$ ($\geq 0$). With $G \equiv BR^{-1} B^{\T} \geq 0$, 
the optimal control $u(t) = -R^{-1} B^{\T} X x(t)$ can be expressed in terms of the unique Hermitian  positive semi-definite (psd) stabilizing solution $X$ of the CARE \cite{biniIM2012numerical,chuYL2005structurepreserving,lancasterR1991solutions,mehrmann1991automomous}:
\bel{eq:care}
\mathcal{C}(X) \equiv A^{\T} X + X A - X G X + H = 0.
\eel
In the paper we shall assume without loss of generality that $B$ and $C^{\T}$ are of full column rank and $R = I_m$, for the sake of simpler notations in later development.

Analogously, for the LQ optimal control of the linear time-invariant control system in discrete-time:
 \[
	 x_{k+1} = A x_k + B u_k, \ \ \ k=0,1, 2, \cdots,
 \]
the corresponding optimal control $u_k = -(R + B^{\T} X B)^{-1} B^{\T} X A x_k$ can be expressed in terms of the unique psd stabilizing solution $X$ of the discrete-time algebraic Riccati equation (DARE) \cite{biniIM2012numerical,chuFLW2004structurepreserving,lancasterR1991solutions,mehrmann1991automomous}:  
\begin{equation}\label{eq:dare}
	\mathcal{D}(X) \equiv 
-X + A^{\T} X (I + GX)^{-1} A + H = 0.  
\end{equation}
Let  $A_0\equiv A,  G_0\equiv G$, and $H_0\equiv H$. 
Assuming that $I_n + G_k H_k$ are nonsingular for $k=0, 1, \cdots$, the SDA for DAREs has three iterative recursions: 
\begin{equation}\label{eq:sda}
	\begin{aligned}[b]
		A_{k+1}& = A_k (I_n + G_k H_k)^{-1} A_k, \\ 
		G_{k+1} &= G_k + A_k (I_n + G_k H_k)^{-1}  G_k A_k^{\T}, \\ 
		H_{k+1} &= H_k + A_k^{\T} H_k(I_n + G_k H_k)^{-1} A_k. 
	\end{aligned}
\end{equation}
We have $A_k \ra 0$, $G_k \ra Y$ (the solution to the dual DARE) and $H_k \ra X$, all quadratically \cite{mehrmann1991automomous} except for the critical case~\cite{huangL2009structured}.  

After the Cayley transform with  nonsingular $A_\gamma := A - \gamma I$ ($\gamma > 0$) and $K_\gamma := A_\gamma^{\T} + H A_\gamma^{-1} G$, the SDA for CAREs \cite{chuYL2005structurepreserving} shares the same formulae \eqref{eq:sda}, with the alternative starting points:
\begin{align}\label{eq:starting-points-care}
	A_0 = I_n + 2\gamma K_\gamma^{-{\T}}, \ \ G_0 = 2\gamma A_\gamma^{-1} G K_\gamma^{-1}, \ \ H_0 = 2\gamma K_\gamma^{-1} H A_\gamma^{-1}. 
\end{align}

Unlikely  to the doubling  formulae~\eqref{eq:sda} for DAREs and CAREs, the SDA has four coupled iteration  recursions for the MARE:
\begin{equation}\label{eq:mare}
XCX - XD - AX + B = 0, 
\end{equation}
where 	$A\in \mathbb{R}^{m\times m}, \quad B,X\in \mathbb{R}^{m\times n}$, $C\in \mathbb{R}^{n\times m}$ 
and $D\in \mathbb{R}^{n\times n}$.   With $a_{ii}$ and $d_{jj}$ respectively being the 
diagonal entries of $A$ and $D$,  for $\gamma \geq \max_{i,j} \{ a_{ii},\, d_{jj} \}$, let  
\begin{align*}
	A_\gamma& := A + \gamma I_m, \qquad  &  D_\gamma &:= D + \gamma I_n,  \\
	W_\gamma &:= A_\gamma - BD_\gamma^{-1} C, \qquad &  V_\gamma &:= D_\gamma - CA_\gamma^{-1} B, \\
	F_0& = I_m - 2\gamma W_\gamma^{-1}, \qquad & E_0 &= I_n - 2\gamma V_\gamma^{-1},  \\
	H_0 &= 2\gamma W_\gamma^{-1} B D_\gamma^{-1}, \qquad  & G_0 &= 2\gamma D_\gamma^{-1} C W_\gamma^{-1}. 
\end{align*}
Assuming that $I_m - H_k G_k$ and $I_n - G_k H_k$ are nonsingular for $k=0, 1, \cdots$, 
the SDA for MAREs has the form:
\begin{equation}\label{eq:maresda}
	\begin{aligned}
		F_{k+1} &= F_k (I_m - H_k G_k)^{-1} F_k, \qquad & E_{k+1}& = E_k (I_n - G_k H_k)^{-1} E_k, \\
		H_{k+1} &= H_k + F_k (I_m-H_k G_k)^{-1} H_k E_k, \qquad & G_{k+1} &= G_k + E_k (I_n-G_k H_k)^{-1} G_k F_k, 
	\end{aligned}
\end{equation}
where  $E_k, F_k \ra 0$ and $H_k \ra X$, $G_k \to Y$ ($Y$ is the unique minimal nonnegative solution to the dual MARE) as $k \ra \infty$. 

MAREs have been considered widely in \cite{guoIM2007doubling, liuX2016complex, xueXL2012accurate, xueL2017highly, guoLX2006structurepreserving,liCKL2013solving,  baiGL2008fast,biniMP2010transforming,biniIP2008fast,chiangCGHLX2009convergence,guo2001nonsymmetric,guoL2000iterative,guoL2010convergence,juang1995existence,mehrmannX2008explicit,wangWL2012alternatingdirectional,chguoN2007iterative,lu2005newton}, usually with $M = \begin{mat}{rr} D & -C \\ -B & A \end{mat}$ being a nonsingular or an  irreducible singular  M-matrix  for   the solvability  of \eqref{eq:mare}. Actually, \eqref{eq:mare} has a unique minimal nonnegative solution in such conditions. Here, a matrix  is  nonnegative  if all its entries are  nonnegative and $X$ is the minimal nonnegative solution if $\widetilde{X}-X$ is nonnegative for all solutions $\widetilde{X}$.

\section{Decoupled form of SDA}\label{sec:decoupled-form-for-sda}

For the classical SDA, when the initial iterates possess low-rank structures, the three  coupled iterates (four for MAREs or two for BSEPs) can be decoupled. This leads to the decoupled form in the dSDA shown  in  this section.  This section contains many tedious but necessary details and the SMWF~\eqref{eq:smwf} will be used  repeatedly.  

\subsection{dSDA for DAREs}\label{ssec:dsda-for-dares}

\subsubsection{New formulation for the first step}\label{sssec:new-formulation-for-the-first-step}

With  the initial values  
$A_0 = A$,
$G_0 = BB^{\T}$ and 
$H_0 = C^{\T} C$, where $I_n + G_0 H_0$ is nonsingular, we are going to reformulate $A_1, G_1$ and $H_1$. 
By the SMWF~\eqref{eq:smwf}, with $Y_0 := B^{\T} C^{\T}$, $U_0 := B$, $V_0 := C^{\T}$, 
$E_0 := I_m + Y_0 Y_0^{\T}$, $F_0 := I_l + Y_0^{\T} Y_0$ and $K_0 := E_0^{-1} Y_0 = Y_0 F_0^{-1}$,   
we have
\begin{align*}
	(I_n + G_0 H_0)^{-1} &= (I_n + BB^{\T} C^{\T} C)^{-1} 
	= I_n - B (I_m + B^{\T} C^{\T} C B)^{-1}  B^{\T} C^{\T} C \\  
	&= I_n - U_0 (I_m + Y_0 Y_0^{\T})^{-1} Y_0 V_0^{\T} 
	\equiv I_n - U_0 K_0 V_0^{\T}. 
\end{align*}
Similarly, with the symmetric $I_m - K_0 Y_0^{\T} = E_0^{-1}$ and $I_l - Y_0^{\T} K_0 = F_0^{-1}$, it holds that  
\begin{align*}
	(I_n+G_0H_0)^{-1} G_0 &= (I_n - U_0 E_0^{-1} Y_0 V_0^{\T}) U_0 U_0^{\T} 
	\equiv U_0 E_0^{-1} U_0^{\T}, \\
	H_0 (I_n+G_0 H_0)^{-1} &= V_0 V_0^{\T} (I_n - U_0 E_0^{-1} Y_0 V_0^{\T}) 
	\equiv V_0 F_0^{-1} V_0^{\T}.
\end{align*}
With $U_1 := A_0 U_0$, $V_1 := A_0^{\T} V_0$, $M_1^{A} := 0 \oplus K_0$, 
$M_1^{G} := I_m \oplus E_0^{-1}$, $M_1^{H} := I_l \oplus F_0^{-1}$, $\widehat{U}_1 := [U_0, U_1]$ and $\widehat{V}_1:=[V_0,V_1]$, some simple calculations produce  
\begin{align*}
	A_1 &= A_0 (I_n + G_0 H_0)^{-1} A_0 
	= A_0 (I_n - U_0 K_0 V_0^{\T}) A_0 
	\equiv A_0^2 -	\widehat{U}_1
	M_1^{A}\widehat{V}_1^{\T}, 
	\\ 
	G_1 &= G_0 + A_0 (I_n + G_0 H_0)^{-1} G_0 A_0^{\T}  
	= U_0 U_0^{\T} + A_0 (U_0 E_0^{-1} U_0^{\T}) A_0^{\T} 
	\equiv \widehat{U}_1 M_1^{G} \widehat{U}_1^{\T}, 
	\\  
	H_1 &= H_0 + A_0^{\T} H_0 (I_n+G_0 H_0)^{-1} A_0^{\T} 
	= V_0 V_0^{\T} + A_0^{\T} (V_0 F_0^{-1} V_0^{\T}) A_0^{\T} 
	\equiv \widehat{V}_1 M_1^{H} \widehat{V}_1^{\T}. 
\end{align*}
Moreover, with  
$Y_1 := \begin{bmatrix}
	0&0\\0&Y_0
\end{bmatrix}
\in \mathbb{R}^{2m\times 2l}$,
it is easy to see that 
\begin{equation}\label{eq:M1}
	M_1^A= M_1^G Y_1= Y_1M_1^H, \ \   
	(M_1^G)^{-1}=I_{2m}+Y_1 Y_1^{\T}, \ \ 
	(M_1^H)^{-1}=I_{2l}+Y_1^{\T} Y_1,
\end{equation}
implying that
\begin{align*}
	A_1 &= A_0^2-\widehat{U}_1
	\left(I_{2m}+Y_1 Y_1^{\T}\right)^{-1} Y_1 
	\widehat{V}_1^{\T},
	\\ 
	G_1 &= \widehat{U}_1 
	\left(I_{2m}+Y_1 Y_1^{\T}\right)^{-1}
	\widehat{U}_1^{\T}, 
	\qquad 
	H_1 =\widehat{V}_1 
	\left(I_{2l}+ Y_1^{\T} Y_1\right)^{-1}
	\widehat{V}_1^{\T}. 
\end{align*}

\subsubsection{New formulation for the second step}\label{sssec:new-formulation-for-the-second-step}

For the $2$nd iteration, with  
\begin{align*}
	T_1 := \widehat{U}_1^{\T} \widehat{V}_1, & \qquad 
	K_1 := \left( I_{2m} + M_1^G T_1 M_1^H T_1^{\T} \right)^{-1} M_1^G T_1 M_1^H,
\end{align*}  
and $I_n+G_1 H_1$ being nonsingular, then by the SMWF~\eqref{eq:smwf}  we deduce that 
\begin{align*}
	&(I_n+G_1 H_1)^{-1} 
	= 
	\left( I_n + \widehat{U}_1 	M_1^G   
		\widehat{U}_1^{\T} \widehat{V}_1
		M_1^H   \widehat{V}_1^{\T}  
	\right)^{-1} 
	\\
	=& 
	I_n - \widehat{U}_1	\left( I_{2m} + M_1^G  T_1 M_1^H  T_1^{\T} \right)^{-1} 
	M_1^G  T_1 M_1^H  \widehat{V}_1^{\T}  
	\equiv  
	I_n - \widehat{U}_1
	K_1 
	\widehat{V}_1^{\T}. 
\end{align*}
Define  $E_1 :=  (M_1^G)^{-1} + T_1 M_1^H  T_1^{\T} $ and $F_1 :=  (M_1^H)^{-1}+ T_1^{\T}  M_1^G  T_1$, then manipulations produce 
\begin{align*}
	& (I_n+G_1 H_1)^{-1} G_1 
	= 	
	\left( I_n - \widehat{U}_1 	K_1 
		\widehat{V}_1^{\T} 
	\right) \widehat{U}_1
	M_1^G  
	\widehat{U}_1^{\T} 
	\\
	=& 
	\widehat{U}_1 \left\{ M_1^G  - \left( I_{2m} + M_1^G  T_1 M_1^H  T_1^{\T} \right)^{-1} 
	M_1^G  T_1 M_1^H  T_1^{\T} M_1^G \right\} \widehat{U}_1^{\T} 
	\\
	=&
	\widehat{U}_1
	\left\{ I_{2m}  - \left( I_{2m} + M_1^G  T_1 M_1^H  T_1^{\T} \right)^{-1} 
	M_1^G  T_1 M_1^H  T_1^{\T}  \right\}
	M_1^G \widehat{U}_1^{\T} 
	\equiv 	
	\widehat{U}_1 E_1^{-1} \widehat{U}_1^{\T},	
\end{align*} 
and 
\begin{align*}
	& H_1 (I_n+G_1 H_1)^{-1}     
	=
	\widehat{V}_1 
	M_1^H  \widehat{V}_1^{\T}   
	\left\{ I_n -  \widehat{U}_1 K_1 \widehat{V}_1^{\T} \right\} 
	\\
	=&
	\widehat{V}_1 M_1^H  \left\{I_{2l} - T_1^{\T} \left(I_{2m} + M_1^G  T_1 M_1^H  T_1^{\T} \right)^{-1} 
	M_1^G  T_1 M_1^H  \right\}  \widehat{V}_1^{\T} 
	\equiv 
	\widehat{V}_1 F_1^{-1} \widehat{V}_1^{\T}. 
\end{align*}
By denoting $U_2:= A_0^2 U_0$, $U_3:= A_0^3 U_0$, $V_2:= (A_0^2)^{\T} V_0$, $V_3:= (A_0^3)^{\T} V_0$, $\widehat{U}_2 := [U_0, U_1, U_2, U_3]$ and $\widehat{V}_2 :=[V_0, V_1, V_2, V_3]$,  we obtain 
\begin{align*}
	A_2 &= A_1 (I_n + G_1 H_1)^{-1} A_1 
	= \left\{ A_0^2 - \widehat{U}_1 M_1^A  \widehat{V}_1^{\T} \right\} 
	\left\{ I_n - \widehat{U}_1 K_1 \widehat{V}_1^{\T} \right\} 
	\left\{ A_0^2 - \widehat{U}_1 M_1^A  \widehat{V}_1^{\T} \right\} 
	\\
	&= A_0^4 - 
	\widehat{U}_2 
	\begin{bmatrix} -M_1^A  T_1^{\T} (I_{2m} - K_1 T_1^{\T}) M_1^A  & M_1^A  (I_{2l} - T_1^{\T} K_1)  
	\\[8pt]  (I_{2m} - K_1 T_1^{\T}) M_1^A  & K_1 \end{bmatrix} 
	\widehat{V}_2^{\T} 
	\\
	&= A_0^4 - \widehat{U}_2 \begin{bmatrix}I_{2m}&-M_1^A T_1^{\T}\\[5pt]  0&I_{2m}\end{bmatrix}
	\begin{bmatrix}M_1^A T_1^{\T} M_1^A &M_1^A \\[5pt]  M_1^A &K_1\end{bmatrix}
	\begin{bmatrix}I_{2l}&0\\[5pt] -T_1^{\T} M_1^A  &I_{2l}\end{bmatrix}
	\widehat{V}_2^{\T} 	
	:= A_0^4 -\widehat{U}_2 M_2^A \widehat{V}_2^{\T}, 
\end{align*}
\begin{align*}
	G_2 
	&=  
	G_1 + A_1 (I_n+G_1 H_1)^{-1} G_1 A_1^{\T} 
	\\
	&= 
	\widehat{U}_1 
	M_1^G  \widehat{U}_1^{\T} 
	+ \left\{ A_0^2 - \widehat{U}_1  M_1^A  \widehat{V}_1^{\T} \right\} 
	\widehat{U}_1 E_1^{-1} \widehat{U}_1^{\T}  
	\left\{ {(A_0^{2})}^{\T} - 
		\widehat{V}_1 
		(M_1^A)^{\T} \widehat{U}_1^{\T} 
	\right\} 
	\\
	&=  
	\widehat{U}_2 
	\begin{bmatrix} M_1^G  + M_1^A  T_1^{\T} E_1^{-1} T_1 (M_1^A)^{\T} & -M_1^A  T_1^{\T} E_1^{-1} \\[8pt] -E_1^{-1} T_1 (M_1^A)^{\T} & E_1^{-1} \end{bmatrix}
	\widehat{U}_2^{\T} 
	\\
	&=\widehat{U}_2 
	\begin{bmatrix} I_{2m} & -M_1^A T_1^{{\T}}\\[5pt] 0&I_{2m}\end{bmatrix}
	\begin{bmatrix} M_1^G  &0 \\[5pt] 0 & E_1^{-1} \end{bmatrix}
	\begin{bmatrix}I_{2m}&0\\[5pt]  -T_1(M_1^A)^{\T} &I_{2m}\end{bmatrix}
	\widehat{U}_2^{\T}
	:= \widehat{U}_2 M_2^G \widehat{U}_2^{\T},
\end{align*}			
and 
\begin{align*}
	H_2 
	&= 
	H_1 + A_1^{\T} H_1 (I_n+G_1 H_1)^{-1} A_1 
	\\
	&=  
	\widehat{V}_1 	M_1^H \widehat{V}_1^{\T} 
	+	\left\{ {(A_0^{2})}^{\T} - \widehat{V}_1 (M_1^A)^{\T} \widehat{U}_1^{\T} \right\}
	\widehat{V}_1 F_1^{-1} \widehat{V}_1^{\T} 
	\left\{ A_0^2 - 
		\widehat{U}_1 
		M_1^A  \widehat{V}_1^{\T} 
	\right\} 
	\\
	&= 
	\widehat{V}_2 	\begin{bmatrix} M_1^H  + (M_1^A)^{\T} T_1F_1^{-1} T_1^{\T} M_1^A  & -(M_1^A)^{\T} T_1F_1^{-1} 
	\\[8pt] -F_1^{-1} T_1^{\T} M_1^A  &F_1^{-1} \end{bmatrix}
	\widehat{V}_2^{\T} 
	\\ 
	&= \widehat{V}_2 
	\begin{bmatrix}I_{2l}&-(M_1^A)^{\T} T_1\\[5pt] 0&I_{2l}\end{bmatrix}
	\begin{bmatrix} M_1^H  & 0 \\[5pt]  0&F_1^{-1}\end{bmatrix} 
	\begin{bmatrix}I_{2l}&0\\[5pt] -T_1^{\T} M_1^A &I_{2l} \end{bmatrix}
	\widehat{V}_2^{\T} 
	:= 
	\widehat{V}_2 	M_2^H \widehat{V}_2^{\T}. 
\end{align*}
By \eqref{eq:M1}, and  the definitions of $E_1, F_1$ and $K_1$, 
we then have 
\begin{align*}
	(M_2^G)^{-1}M_2^A&=
	\begin{bmatrix} I_{2m} &0\\[5pt] T_1(M_1^A)^{{\T}}&I_{2m} \end{bmatrix}
	\begin{bmatrix}(M_1^G)^{-1} &\\[5pt] &E_1\end{bmatrix}
	\begin{bmatrix}M_1^AT_1^{\T} M_1^A&M_1^A\\[5pt] M_1^A&K_1\end{bmatrix}
	\begin{bmatrix}I_{2l}&0\\[5pt] -T_1^{\T} M_1^A &I_{2l}\end{bmatrix}
	\\
	&\equiv\begin{bmatrix}
		0&Y_1\\Y_1&T_1
	\end{bmatrix} := Y_2\in\mathbb{R}^{4m\times 4l},
\end{align*}
and 
\begin{align*}
	M_2^A(M_2^H)^{-1}&=
	\begin{bmatrix} I_{2m} &-M_1^A T_1^{\T}\\[5pt] 0&I_{2m}\end{bmatrix}
	\begin{bmatrix}M_1^A T_1^{\T} M_1^A&M_1^A\\[5pt] M_1^A&K_1 \end{bmatrix}
	\begin{bmatrix}(M_1^H)^{-1} &\\[5pt] &F_1 \end{bmatrix}
	\begin{bmatrix}I_{2l}&(M_1^A)^{\T} T_1\\[5pt] 0 &I_{2l}\end{bmatrix}\equiv Y_2,
\end{align*}
implying that $M_2^A=M_2^G Y_2 = Y_2M_2^H$.
Furthermore, it follows from \eqref{eq:M1}, the definition of $E_1$ and $M_1^H + Y_1^{\T} M_1^A = I_{2l}$ 
that 
\begin{align*}
	(M_2^G)^{-1}& = \begin{bmatrix}
		I_{2m} & 0\\[5pt]  T_1(M_1^A)^{\T} &I_{2m}
	\end{bmatrix}
	\begin{bmatrix}
		(M_1^G)^{-1}&0\\[5pt] 0& E_1
	\end{bmatrix}
	\begin{bmatrix}
		I_{2m}&M_1^AT_1^{\T} \\[5pt] 0 &I_{2m}
	\end{bmatrix}
	=
	\begin{bmatrix}
		(M_1^G)^{-1} & Y_1 T_1^{\T} \\[5pt] T_1 Y_1^{\T} & E_1+T_1Y_1^{\T} M_1^A T_1^{\T}
	\end{bmatrix}\\
	&=\begin{bmatrix}
		(M_1^G)^{-1}&Y_1T_1^{\T} \\[5pt]  T_1 Y_1^{\T} & (M_1^G)^{-1} + T_1 (M_1^H + Y_1^{\T} M_1^A)T_1^{\T} 
	\end{bmatrix}
	=\begin{bmatrix}
		(M_1^G)^{-1} & Y_1 T_1^{\T}\\[5pt]  T_1Y_1^{\T} &(M_1^G)^{-1}+T_1 T_1^{\T}
	\end{bmatrix},
\end{align*}
indicating that 
\[
	(M_2^G)^{-1} - Y_2Y_2^{\T} = \begin{bmatrix}
		(M_1^G)^{-1}-Y_1 Y_1^{\T} & 0\\ 0 &(M_1^G)^{-1} - Y_1 Y_1^{\T}
	\end{bmatrix}\equiv I_{4m}.
\]
Similarly, $(M_2^H)^{-1}-Y_2^{\T} Y_2=I_{4l}$.
Consequently, we have the following result.

\begin{lemma}\label{lm:DAREiteration2}
For the SDA in \eqref{eq:sda}, with $I_n + G_1 H_1$ being nonsingular, we have the following decoupled  forms: 
\begin{align*}
	A_2=A_0^4-\widehat{U}_2 
	(I_{4m}+Y_2Y_2^{\T})^{-1}Y_2
	\widehat{V}_2^{\T}, \ \ 
	G_2= \widehat{U}_2 (I_{4m}+Y_2Y_2^{\T})^{-1}\widehat{U}_2^{\T}, \ \ 
	H_2&= \widehat{V}_2 (I_{4l}+Y_2^{\T} Y_2)^{-1} \widehat{V}_2^{\T}.
\end{align*}
\end{lemma}

\subsubsection{Decoupled recursions for DAREs}\label{sssec:decoupled-recursions-for-dares}

Similarly and recursively, we have the following result for $A_k, G_k$ and  $H_k$.

\begin{theorem}[Decoupled formulae  of the dSDA for DAREs]\label{thm:DAREiterationk}
	Let $U_j:= A_0 U_{j-1}$ ($U_0 = B$) and $V_j := A_0^{\T} V_{j-1}$ ($V_0 = C^{\T}$)  for $j \geq 1$. 
	Assume that $I_n + G_k H_k$ are nonsingular for $k \geq 0$. 
	For all $k\ge 2$, the SDA \eqref{eq:sda} produces 
		\begin{align*}
			&A_k=A_0^{2^k}-
			\widehat{U}_k 			
			\left(I_{2^km}+Y_kY_k^{\T}\right)^{-1}Y_k
			\widehat{V}_k^{\T},   
			\\
			& G_k= \widehat{U}_k (I_{2^k m}+Y_kY_k^{\T})^{-1} \widehat{U}_k^{\T},  \qquad \qquad 
			H_k= \widehat{V}_k (I_{2^k l}+Y_k^{\T} Y_k)^{-1} \widehat{V}_k^{\T}, 
		\end{align*}
		where $Y_k := \begin{bmatrix}
			0&Y_{k-1}\\Y_{k-1}&T_{k-1}
		\end{bmatrix}\in\mathbb{R}^{2^km\times2^kl}$, 
		$\widehat{U}_k := [U_0,U_1,\cdots ,U_{2^k-1}]$, 
		$\widehat{V}_k := [V_0,V_1,\cdots ,V_{2^k-1}]$ and 
	$T_{k-1} := \widehat{U}_{k-1}^{\T} \widehat{V}_{k-1}$ with  
	$Y_1 = \begin{bmatrix} 0 & 0 \\ 0 & B^{\T} C^{\T} \end{bmatrix}\in \mathbb{R}^{2m\times 2l}$.  
\end{theorem}
\begin{proof}
We will prove the result by induction. 	By Lemma~\ref{lm:DAREiteration2} the result is valid  for $k=2$. Now assume that the result holds for $j>2$, then by the SMWF~\eqref{eq:smwf} we have 
	\begin{align*}
		(I_n +G_j H_j)^{-1} &= I_n -\widehat{U}_{j} 
		K^{-1} T_j N \widehat{V}_{j}^{\T},
		\\
		(I_n +G_j H_j)^{-1} G_j &= \widehat{U}_{j} K^{-1}  \widehat{U}_{j}^{\T},
		\qquad \qquad 
		H_j(I_n +G_j H_j)^{-1}=  \widehat{V}_{j} L^{-1}  \widehat{V}_{j}^{\T},
	\end{align*}
	where $M=\left(I_{2^j m}+Y_jY_j^{\T}\right)^{-1}$, $N=\left(I_{2^j l}+Y_j^{\T} Y_j\right)^{-1}$,   $K=I_{2^j m} + Y_j Y_j^{\T} + T_j N T_j^{\T}$ and  $L =  I_{2^jl}+Y_j^{\T} Y_j + T_j^{\T} M T_j$. 
Define   $Z_1:= \left(I_{2^j m}+Y_jY_j^{\T}\right)^{-1} Y_j T_j^{\T}$ and 	$Z_2 := \left(I_{2^j l}+Y_j^{\T} Y_j\right)^{-1} Y_j^{\T} T_j$, then  by  \eqref{eq:sda} it holds that  
\begin{align*}
	A_{j+1}
	&=
	\begin{multlined}[t]
		A_0^{2^{j+1}} - 
		\widehat{U}_{j+1}
		\begin{bmatrix}
			I_{2^j m}& -Z_1  \\[5pt]  0 & I_{2^j m}
		\end{bmatrix}
		\begin{bmatrix}
			Z_1 M Y_j 
			& M Y_j \\[5pt]
			M Y_j
			&K^{-1} T_j N 
		\end{bmatrix} 
		\begin{bmatrix}
			I_{2^jl} & 0 \\[5pt]  -Z_2^{\T} & I_{2^jl}
		\end{bmatrix}	
		\widehat{V}_{j+1}^{\T}, 
	\end{multlined}
	\\
	G_{j+1}    
	&=
	\widehat{U}_{j+1}	\begin{bmatrix}
		M + Z_1 K^{-1} Z_1^{\T}		& -Z_1 K^{-1}  
		\\[5pt] 
		-K^{-1} Z_1^{\T} 
		& K^{-1}
	\end{bmatrix}		\widehat{U}_{j+1}^{\T}
	\\
	&=
	\widehat{U}_{j+1}
	\begin{bmatrix}
		I_{2^jm}&-Z_1\\[5pt] 0&I_{2^j m}
	\end{bmatrix}
	\begin{bmatrix}
		M  &0
		\\[5pt]
		0& K^{-1}
	\end{bmatrix}
	\begin{bmatrix}
		I_{2^j m}&0\\[5pt] -Z_1^{\T} &I_{2^j m}
	\end{bmatrix}
	\widehat{U}_{j+1}^{\T}, 
\end{align*} 
and 
\begin{align*}
	H_{j+1}
	&=
	\widehat{V}_{j+1}
	\begin{bmatrix}
		N  + Z_2 L^{-1} Z_2^{\T}	& -Z_2L^{-1}\\[5pt] 
		-L^{-1} Z_2^{\T} & L^{-1}
	\end{bmatrix}
	\widehat{V}_{j+1}^{\T},
	\\
	&=
	\widehat{V}_{j+1}
	\begin{bmatrix}
		I_{2^j l}&-Z_2\\[5pt] 0&I_{2^j l}
	\end{bmatrix}
	\begin{bmatrix}
		N &0\\[5pt] 0&L^{-1}
	\end{bmatrix}
	\begin{bmatrix}
		I_{2^j l}&0\\[5pt] -Z_2^{\T}&I_{2^j l}
	\end{bmatrix}\widehat{V}_{j+1}^{\T}. 
\end{align*}

Define  $Y_{j+1} := \begin{bmatrix}
	0&Y_j\\[5pt] Y_j&T_{j}
	\end{bmatrix}\in \mathbb{R}^{2^{j+1}m \times 2^{j+1}l}$, 
then we have  
\[
	\begin{bmatrix}
		I_{2^jm} & 0\\[5pt] Z_1^{\T} & I_{2^jm}
	\end{bmatrix}
	\begin{bmatrix}
		I_{2^j m}+Y_jY_j^{\T} &0\\[5pt]0&K
	\end{bmatrix}
	\begin{bmatrix}
		I_{2^jm} & Z_1\\[5pt] 0 & I_{2^jm}
	\end{bmatrix}
	\equiv I_{2^{j+1}m}+Y_{j+1}Y_{j+1}^{\T},
\]
leading to 
	$G_{j+1}=
	\widehat{U}_{j+1}		\left(I_{2^{j+1}m}+Y_{j+1}Y_{j+1}^{\T}\right)^{-1}
	\widehat{U}_{j+1}^{\T}$. 
	Similarly, we can verify analogous formulae for $A_{j+1}$ and $H_{j+1}$. The proof by induction is complete. 
\end{proof}

\begin{remark} \label{rem1}
	Theorem~\ref{thm:DAREiterationk} decouples the original SDA \eqref{eq:sda}, implying that only the iteration for $H_k$ is required for the solution of the DARE \eqref{eq:dare}. 
	This eliminates the difficulty in \cite{liCLW2013solving}, in which the implicit recursion in $A_k$ increases the flop counts exponentially.  
	The decoupled formulae are simple and elegant, with the updating recursion for $Y_k$ nontrivial. The resulting dSDA is obviously equivalent to a projection method, with the corresponding Krylov subspace $\mathcal{K}_{2^k-1} (A^{\T}, C^{\T}) \equiv [C^{\T}, A^{\T} C^{\T}, \cdots, (A^{\T})^{2^k-1} C^{\T}]$ and the coefficient matrix $(I+Y_k^{\T} Y_k)^{-1}$ being the solution of the corresponding projected equation. Note that the SDA (and the equivalent dSDA) has been proved to converge \cite{linX2006convergence}, assuming that $I_n + G_k H_k$ are nonsingular for $k \geq 0$. In contrast, the projected equations in Krylov subspace methods are routinely assumed to be solvable \cite{heyouniJ2009extended,jbilou2003block,jbilou2006arnoldi,simoncini2016analysis,simonciniSM2014two}.  However, with round-off errors, the Krylov subspaces may lose linear independence, requiring a remedy in a truncation process. Also, near convergence, new additions to the Krylov subspaces play lesser parts in the approximate solution, implying that the coefficient matrix $(I+Y_k^{\T} Y_k)^{-1}$ has relatively smaller components in the lower right corner. Without truncation, the ill-conditioned coefficient matrix, as an inverse, will be difficult to compute as $k$ increases. Limited by the article space, many practical compute issues will be treated in another paper, where we shall develop a novel truncation technique. 
\end{remark}

\begin{remark}\label{rk:complex-conjugate}
	For the complex DARE: 
	\begin{align*}
		-X + A^{\HH} X (I + GX)^{-1} A + H = 0, 
	\end{align*}
	the above decoupled form of the  SDA proves valid, with  the   $(\cdot)^{\T}$ replaced  by $(\cdot)^{\HH}$.  Such comment applies for the results in the subsequent sections. 
\end{remark}

\subsection{dSDA for CAREs}\label{ssec:dsda-for-cares}
Note that in~\eqref{eq:starting-points-care}, the starting points for CAREs are different  from those for DAREs, and to get $A_0, G_0, H_0$ we need to compute $K_{\gamma}$ at first.  For CAREs, alternatively with $U_0 := A_\gamma^{-1} B$, $V_0 := A_\gamma^{-{\T}} C^{\T}$and $Y_0 := B^{\T} V_0$, then  
by  the SMWF~\eqref{eq:smwf} we get 
\[
	K_{\gamma}^{-1}=A_{\gamma}^{-{\T}}-V_0\left(I_l + Y_0^{\T} Y_0\right)^{-1}Y_0^{\T} U_0^{\T},
\]
leading to 
\begin{equation}\label{eq:caresda0}
	A_0 = (I_n+2\gamma A_\gamma^{-1}) - 2\gamma U_0  Y_0 F_0^{-1} V_0^{\T}, \ \ 
	G_0 = 2\gamma U_0 E_0^{-1} U_0^{\T}, \ \ 
    H_0 = 2\gamma V_0 F_0^{-1} V_0^{\T}, 
\end{equation}
where $E_0 := I_m + Y_0 Y_0^{\T}$, $F_0 := I_l + Y_0^{\T} Y_0$ 
satisfy $Y_0 F_0^{-1} = E_0^{-1} Y_0$ and $Y_0^{\T} E_0^{-1} = F_0^{-1} Y_0^{\T}$.

Defining   $T_0 := U_0^{\T} V_0$, 
$K := (E_0 + 4\gamma^2 T_0 F_0^{-1} T_0^{\T})^{-1}$ and $L := (F_0+4\gamma^2T_0^{\T} E_0^{-1} T_0)^{-1}$, the SMWF~\eqref{eq:smwf} again implies 
\begin{equation}\label{eq:curlK}
\begin{aligned}
	(I_n + G_0 H_0)^{-1} 
	&= ( I_n + 4 \gamma^2 U_0 E_0^{-1} T_0 F_0^{-1} V_0^{\T} )^{-1}=I_n-4\gamma^2U_0KT_0F_0^{-1}V_0^{\T},\\
	(I_n + G_0 H_0)^{-1} G_0 &=2\gamma ( I_n + 4 \gamma^2 U_0 E_0^{-1} T_0 F_0^{-1} V_0^{\T} )^{-1}U_0E_0^{-1}U_0^{\T}=2\gamma U_0 K  U_0^{\T},\\
	H_0 (I_n + G_0 H_0)^{-1} &=2\gamma V_0F_0^{-1}V_0^{\T}( I_n + 4 \gamma^2 U_0 E_0^{-1} T_0 F_0^{-1} V_0^{\T} )^{-1}= 2\gamma V_0 L V_0^{\T}. 
\end{aligned}
\end{equation}
Denote  $\widetilde A_{\gamma}:= ( I_n+2\gamma A_\gamma^{-1})$,  
$U_1:= \widetilde A_{\gamma} U_0,$  
$V_1:= \widetilde A_{\gamma}^{\T} V_0$ and $Y_1 := \begin{mat}{cc} 0 & Y_0 \\[3pt] Y_0 & 2\gamma T_0 \end{mat}\in \mathbb{R}^{2m\times 2l}$, with similar notations $\widehat{U}_1$ and $\widehat{V}_1$ as in the previous section (here $U_1:=\widetilde{A}_{\gamma}U_0$ and $V_1:=\widetilde{A}_{\gamma}^{\T}V_0$) and the help of \eqref{eq:sda}, ~\eqref{eq:caresda0} and \eqref{eq:curlK}, some   manipulations yield 
\begin{equation}\label{eq:G1-care}
	\begin{aligned}
	G_1 
	&=
	2\gamma
	\widehat{U}_1
	\begin{bmatrix}
		E_0^{-1} + 4\gamma^2 Y_0 F_0^{-1} T_0^{\T} K  T_0 F_0^{-1} Y_0^{\T} & -2\gamma Y_0 F_0^{-1} T_0^{\T} K\\[5pt]
		-2\gamma KT_0 F_0^{-1} Y_0^{\T} & K
	\end{bmatrix} 
	\widehat{U}_1^{\T} 
	\\ 
	&=
	2\gamma 
	\widehat{U}_1
	\begin{bmatrix}
		E_0&2\gamma Y_0T_0^{\T} \\[5pt] 2\gamma T_0Y_0^{\T} &E_0+4\gamma^2 T_0 T_0^{\T}
	\end{bmatrix}^{-1} 
	\widehat{U}_1^{\T}
	\equiv 
	2\gamma 
	\widehat{U}_1
	\left(I_{2m}+Y_1Y_1^{\T}\right)^{-1} 
	\widehat{U}_1^{\T},
\end{aligned}
\end{equation}
\begin{align*}
	H_1 
	&= 
	2\widehat{V}_1
	\begin{bmatrix}
		F_0^{-1} + 4\gamma^2 F_0^{-1} Y_0^{\T} T_0  L T_0^{\T} Y_0 F_0^{-1} & -2\gamma F_0^{-1} Y_0^{\T} T_0  L \\[5pt]
		-2\gamma  L T_0^{\T} Y_0 F_0^{-1} &  L
	\end{bmatrix}  
	\widehat{V}_1^{\T}  
	\\ 
	&= 
	2\gamma 
	\widehat{V}_1
	\begin{bmatrix}
		F_0 & 2\gamma Y_0^{\T} T_0 \\[5pt] 2\gamma T_0^{\T} Y_0 & F_0 + 4 \gamma^2 T_0^{\T} T_0
	\end{bmatrix}^{-1} 
	\widehat{V}_1^{\T}  
	\equiv 
	2\gamma 
	\widehat{V}_1
	\left(I_{2l}+Y_1^{\T} Y_1\right)^{-1}
	\widehat{V}_1^{\T}, 
\end{align*}
\begin{align*}
	A_1 
	&= 
	\widetilde A_{\gamma}^2  
	-  2\gamma 
	\widehat{U}_1
	\begin{bmatrix}
		-2\gamma Y_0 F_0^{-1}T_0^{\T} K Y_0&Y_0 L\\[5pt] KY_0&2\gamma K T_0F_0^{-1}
	\end{bmatrix}
	\widehat{V}_1^{\T}  
	\\
	&= 
	\widetilde A_{\gamma}^2 
	-  2\gamma  
	\widehat{U}_1
	Y_1 \begin{bmatrix}
		F_0^{-1} + 4\gamma^2 F_0^{-1} Y_0^{\T} T_0  L T_0^{\T} Y_0 F_0^{-1}
		& -2\gamma F_0^{-1} Y_0^{\T} T_0  L 
		\\[5pt] -2\gamma  L T_0^{\T} Y_0 F_0^{-1} &  L
	\end{bmatrix}
	\widehat{V}_1^{\T} 
	\\ 
	&\equiv
	\widetilde A_{\gamma}^2 
	-  2\gamma 
	\widehat{U}_1
	Y_1\left(I_{2l}+Y_1^{\T} Y_1\right)^{-1} 
	\widehat{V}_1^{\T} 
	\equiv 
	\widetilde A_{\gamma}^2 
	-  2\gamma
	\widehat{U}_1
	\left(I_{2m}+Y_1 Y_1^{\T}\right)^{-1}Y_1 
	\widehat{V}_1^{\T}. 
\end{align*}

Obviously, $A_1$, $G_1$ and $H_1$ are decoupled. 
Rewriting the symbols in the recursions, we subsequently get the following decoupled  result for the SDA.

\begin{theorem}[Decoupled formulae of the dSDA for CAREs]\label{thm:CAREiterationk}
	Denote $U_j := \widetilde A_{\gamma} U_{j-1}$ and  \\ $V_j := \widetilde A_{\gamma}^{\T} V_{j-1}$ for $j\geq 1$. Assume that $I_n + G_k H_k$ are nonsingular for $k \geq 0$. For all $k\ge 1$, the SDA produces  
	\begin{equation}\label{eq:care-iteration} 
		\begin{aligned}
			A_k
			&=
			\widetilde A_{\gamma}^{2^k}-2\gamma 
			\widehat{U}_k
			\left(I_{2^k m}+Y_kY_k^{\T}\right)^{-1}Y_k
			\widehat{V}_k^{\T},
			\\
			G_k
			&=
			2\gamma 
			\widehat{U}_k
			\left(I_{2^k m}+Y_k Y_k^{\T}\right)^{-1}
			\widehat{U}_k^{\T}, 
			\qquad 
			H_k
			=
			2\gamma 
			\widehat{V}_k
			\left(I_{2^k l}+Y_k^{\T} Y_k\right)^{-1}
			\widehat{V}_k^{\T},
		\end{aligned}
	\end{equation}
	where  $Y_k=\begin{bmatrix}
		0&Y_{k-1}\\[5pt]Y_{k-1}&2\gamma T_{k-1}
	\end{bmatrix}\in \mathbb{R}^{2^k m \times 2^k l}$, 
	$\widehat{U}_k := [U_0,U_1,\cdots ,U_{2^k-1}]$, 
	$\widehat{V}_k := [V_0,V_1,\cdots ,V_{2^k-1}]$ and 
	$T_k := \widehat{U}_k^{\T} \widehat{V}_k$,  
	with $U_0 = A_\gamma^{-1} B$, $V_0 = A_\gamma^{-{\T}} C^{\T}$, $Y_0 = B^{\T} A_\gamma^{-{\T}} C^{\T}$   
	and $T_0 = U_0^{\T} V_0$. 
\end{theorem}

As the dSDA for DAREs, the three formulae in the iteration are all decoupled. To solve CAREs while monitoring $\| H_k - H_{k-1}\|$ or the normalized residual  for  convergence control,  there is no need  to compute $A_k$ or $G_k$. Comments analogous to those in Remarks~\ref{rem1} and \ref{rk:complex-conjugate} for the dSDA for DAREs also hold for the dSDA for CAREs in Theorem~\ref{thm:CAREiterationk}, with $\widetilde{A}_\gamma$ in place of $A$. 

Theoretically, we can employ the decoupled formulae~\eqref{eq:care-iteration} to compute $H_k$ which approximates  the solution $X$ for the CAREs. However, without truncation the size of $Y_k$ will grow exponentially and the lower right corner of the  kernel $(I+Y_k Y_k^{\T})^{-1}$  will diminish fast.  Hence, how to incorporate truncation in the dSDA is  a crucial issue. We shall solve the associated problems in a companion paper, in which a practical dSDA with truncation strategy will be presented and analyzed.  For a taste of what is to come, the truncation strategy is summarized in the follow diagram.

\newcommand\red[1]{{\color{red}#1}}
\def\trunc{\text{truncation}}
\begin{figure}[h]
	\centering
	\begin{tikzcd}
		\boxed{\red{H_0}} \arrow[r, red]	&\boxed{\red{H_1}} \arrow[r] \arrow[d, red, "\trunc" red]  & H_2 \arrow[r]                                 & H_3 \arrow[r]                                 & H_4 \arrow[r]                                 & H_5 \arrow[r]             & \cdots 
		\\
		&	\boxed{\red{H_1^{(1)}}} \arrow[r, red]                & \red{H_2^{(1)}} \arrow[r] \arrow[d, red, "\trunc" red] & H_3^{(1)} \arrow[r]                           & H_4^{(1)} \arrow[r]                           & H_5^{(1)} \arrow[r]       & \cdots 
		\\
		&	                                         & \boxed{\red{H_2^{(2)}}} \arrow[r, red]                     & \red{H_3^{(2)}} \arrow[r] \arrow[d, dashed, red, "\trunc" red] & H_4^{(2)} \arrow[r]      & H_5^{(2)} \arrow[r]       & \cdots 
		\\
		&	                                         &                                               & \boxed{\red{H_k^{(k)}}} \arrow[r, red]                     & \red{H_{k+1}^{(k)}} \arrow[r] \arrow[d, red, "\trunc" red] & H_{k+2}^{(k)} \arrow[r]       & \cdots 
		\\
		&	                                         &                                               &                                               & \boxed{\red{H_{k+1}^{(k+1)}}} \arrow[r, red] & \red{H_{k+2}^{(k+1)}} \arrow[r] \arrow[d, dashed, red, "\trunc" red] & \cdots \\
		&	                                        	&&&&\quad&                                                                                
	\end{tikzcd}
	\caption{Truncation in dSDA}
	\label{fig:truncation-routine:}
\end{figure}
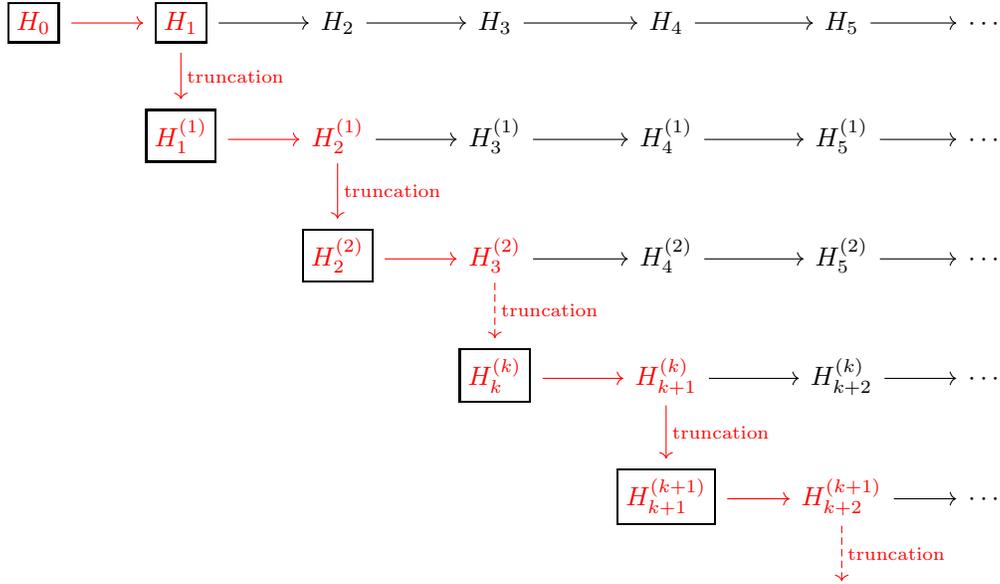

On the first row, we have $H_k$ from the dSDA without truncation. After $H_1$ is computed, it may be truncated to $H_1^{(1)}$ and then it may generate $H_k^{(1)}$ by the dSDA without truncation. In general, $H_k^{(k)}$ generates $H_j^{(k)}$ ($j>k$) by the dSDA without truncation, and $H_{k+1}^{(k)}$ is truncated to $H_{k+1}^{(k+1)}$. Notice that $H_j^{(k)}$ ($j>k$) on any row enjoy the support of the rich existing theory of the dSDA (also SDA). In the companion paper, we shall analyzed the dSDA with truncation extending the results of the dSDA (and SDA). We shall also produce the formula for the short-cut from $H_k^{(k)}$ to $H_{k+1}^{(k+1)}$, without going through $H_{k+1}^{(k)}$.

\subsection{Bethe-Salpeter eigenvalue problem}
In~\cite{guoCL2019doubling}, the SDA is extended to solve the (discretized) BSEP, a Hamiltonian-like eigenvalues problem, where only two iterative formulae are computed instead of three in CAREs. The question is whether  the proposed  dSDA can be generalized to the BSEP.  We give the results in this section.  

Consider the following discretized BSEP:
\begin{equation}  \label{eq:bsevp}
Hx\equiv
\begin{bmatrix}
 \ \ A & \ \ B \\ 
-\overline{B}  & -\overline{A}%
\end{bmatrix}x	
=\lambda x,
\end{equation}
for $x \neq 0$, where $A, B \in\mathbb{C}^{n\times n}$ satisfy
$A^{\HH}=A,\ B^{\T}=B$. For problem \eqref{eq:bsevp} any eigenvalue $\lambda$ appears in quadruplets $\{\pm \lambda, \pm \overline{\lambda}\}$ (except for the degenerate cases when $\lambda$ is purely real or imaginary) and thus shows Hamiltonian-like structure. When applying the SDA to the BSEP~\eqref{eq:bsevp}, by assuming that $I_n-\ol F_kF_k$ are nonsingular for $k \geq 0$ it generates the following iterations: 
\begin{align*}
	E_{k+1}=E_k(I_n-\ol F_kF_k)^{-1}E_k, 
	\ \ \ 
	F_{k+1}=F_k+\ol E_kF_k(I_n-\ol F_kF_k)^{-1}E_k,
\end{align*}
where 
$E_0=I_n-2\alpha\ol R^{-1}(\alpha I_n -A)^{-1}$ and  
	$F_0=-2\alpha(\alpha I_n - \ol A)^{-1}\ol B\ol R^{-1}(\alpha I_n-A)^{-1}$, with 
$R=I_n-(\alpha I_n-\ol A)^{-1}\ol B(\alpha I_n-A)^{-1}B$.

 When initially $B=L_BL_B^{\T}$ with $L_B\in\bbC^{n\times p}$ and $p\le n$, by defining 
 $V_0:=(\alpha I_n -\ol A)^{-1}\ol L_B$, $Y_0:=L_B^{\T}(\alpha I_n -\ol A)^{-1}\ol L_B>0$ and  
$A_{\alpha}:=I_n -2\alpha(\alpha I_n -A)^{-1}$, we have 
\begin{align*}
	\ol R^{-1}&=I_n +\ol V_0( I_p -Y_0 Y_0^{\T})^{-1}Y_0 L_B^{\HH},
	\\
	E_0&=A_{\alpha}-2\alpha \ol V_0 (I_p - Y_0Y_0^{\T})^{-1} Y_0 V_0^{\T},
	\ \ 
	F_0= -2\alpha V_0(I_p -Y_0^{\T} Y_0)^{-1}V_0^{\T}.
\end{align*}
Denote $T_0:=V_0^{\HH}V_0$ and since $Y_0$ is Hermitian, we get 
\begin{align*}
	(I -\ol F_0F_0)^{-1}
	&=
	\begin{multlined}[t]
		I +4\alpha^2\ol V_0 \left\{ I_p - 4\alpha^2(I_p -Y_0Y_0^{\T})^{-1} T_0 (I_p -Y_0^{\T} Y_0)^{-1} \ol T_0 \right\}^{-1} 
		\\
		\cdot (I_p -Y_0Y_0^{\T})^{-1} T_0 (I_p -Y_0^{\T} Y_0)^{-1} V_0^{\T}, 
	\end{multlined}
	\\	
	F_0(I -\ol F_0 F_0)^{-1}
	&=-2\alpha V_0 \left\{I_p -Y_0^{\T}Y_0 - 4\alpha^2 \ol T_0 (I_p -Y_0Y_0^{\T})^{-1}T_0 \right\}^{-1} V_0^{\T}. 
\end{align*}
Furthermore,  by defining  $V_1:=\ol A_{\alpha}V_0$, $\widehat{V}_1 := [V_0, V_1]$ and 
	$Y_1 = \begin{bmatrix}
		0 & Y_0 \\ Y_0 & -2\alpha T_0
	\end{bmatrix}$,
some manipulations similar to those in \eqref{eq:G1-care} yield   
\begin{align*}
	E_1=A_{\alpha}^2-2\alpha \overline{\widehat{V}}_1
	(I_{2p} -Y_1 Y_1^{\T})^{-1} Y_1 \widehat{V}_1^{\T}, 
	\ \ \ 
	F_1&=-2\alpha \widehat{V}_1 
	(I_{2p} -Y_1^{\T} Y_1)^{-1} \widehat{V}_1^{\T}. 
\end{align*}

From the above discussions, we know that the initial $E_0, F_0$ and the first iterates $E_1, F_1$ possess similar structures as  those in the dSDA for  CAREs. Thus with  the similar manipulations, where the SMWF~\eqref{eq:smwf} is applied, we eventually deduce the dSDA for the BSEP, as stated in the following theorem without proof.

\begin{theorem}[Decoupled formulae of the dSDA for the BSEP]\label{thm:decoupled-formulae-of-the-dsda-for-the-bethe-salpeter-eigenvalue-problem}
	Let $V_j = \overline{A}_{\alpha} V_{j-1}$ for $j\ge 1$. 
	Assume that $I_n-\ol F_kF_k$ are nonsingular for $k \geq 0$. 
	Then for all $k\ge 1$ the SDA produces 
	\begin{align*}
		E_k
		& =
		A_{\alpha}^{2^k} - 2\alpha 			
		[\overline{V}_0,  \overline{V}_1, \cdots, \overline{V}_{2^k-1} ] 
		\left(I - Y_k Y_k^{\T}\right)^{-1} Y_k 
		[V_0,  V_1, \cdots, V_{2^k-1} ]^{\T},
		\\
		F_k
		&=
		-2\alpha 
		[V_0,  V_1, \cdots, V_{2^k-1} ] 
		\left(I - Y_k^{\T} Y_k\right)^{-1}
		[V_0,  V_1, \cdots, V_{2^k-1} ]^{\T},
	\end{align*}
where $Y_k = \begin{bmatrix}
	0 & Y_{k-1} \\[5pt]  Y_{k-1} & -2\alpha T_{k-1}
\end{bmatrix}$ with $T_{k-1} = [V_0, V_1, \cdots, V_{2^{k-1}-1}]^{\HH}
[V_0, V_1, \cdots, V_{2^{k-1}-1}]$. 
\end{theorem}

\begin{remark}\label{rk:BS-problem}
Assume that there is no purely imaginary nor zero eigenvalues for $H$ and let  
$H X = X \Lambda$ with $X\in \mathbb{C}^{2n \times n}$ and  all eigenvalues of  $\Lambda \in \mathbb{C}^{n\times n}$ in the  interior of the  left-half  plane. 
Write $X=[X_1^{\T}, X_2^{\T}]$ 
with $X_1 \in \mathbb{C}^{n\times n}$ and choose $\alpha >0$, it then holds that $\lim_{k\to \infty} E_k  =0$ and $\lim_{k\to \infty} F_k = -X_2 X_1^{-1}$. With some simple but tedious computations,  we can further show that $\lim_{k\to \infty}(I- P)[I,  -F_k^{\T}]$
and  $\lim_{k\to \infty} \sin\Theta(X_2 X_1^{-1}, F_k) =0$, where $P$ is the 
projection matrix of $[I,  (X_2 X_1^{-1})^{\T}]$ and 
\begin{align*} 
	\Theta(W, Z) = \arccos\left[
\left(I+\overline{Z} Z\right)^{-1/2} \left(I-\overline{Z} W\right) 
\left(I+\overline{W}W\right)^{-1} \left(I-\overline{W} Z\right) 
\left(I+\overline{Z}Z\right)^{-1/2} \right]^{1/2}.
\end{align*}
Obviously, once obtaining $F_k$ we can approximate all eigenvalues of $\Lambda$ by those of 
\begin{align*}
	H_k=
	[I,  -F_k^{\HH}] 
H [I, -F_k^{\T}]^{\T}
(I + F_k^{\HH} F_k)^{-1}.
\end{align*} 
In \cite{bennerKK2016reduced} it was claimed,   in quantum chemistry  and modern material science, that the matrix  $B$ is of low-rank in some  large-scale discretized BSEPs. Then by Theorem~\ref{thm:decoupled-formulae-of-the-dsda-for-the-bethe-salpeter-eigenvalue-problem}, $F_k$ is of low-rank thus providing further  possibilities for the solution of large-scale BSEPs.   
\end{remark}

\subsection{dSDA for MAREs}\label{ssec:dsda-for-mares}

Assume that $B$ and $C$ are of low rank and possess the full rank factorizations  
$B = B_l B_r^{\T}$ and $C = C_l C_r^{\T}$, where $B_l\in \mathbb{R}^{m\times m_1}, 
B_r \in \mathbb{R}^{n\times m_1}$, $C_l \in \mathbb{R}^{n\times n_1}, C_r\in \mathbb{R}^{m\times n_1}$.  
Denote  
\begin{equation}\label{eq:sym}
	\begin{aligned}
		&		Y_0 := B_r^{\T} D_\gamma^{-1} C_l, \ \ \ Z_0 := C_r^{\T} A_\gamma^{-1} B_l, & \\
		& U_0 := A_\gamma^{-1} B_l, \ \ V_0 := A_\gamma^{-{\T}} C_r, \ \ \ 
		W_0 := D_\gamma^{-1} C_l, \ \ Q_0 := D_\gamma^{-{\T}} B_r, & \\
		&T_0 := Q_0^{\T} W_0 = B_r^{\T} D_\gamma^{-2} C_l, \ \ \ 
		S_0 := V_0^{\T} U_0 = C_r^{\T} A_\gamma^{-2} B_l, & \\
		& \widetilde{A}_\gamma := I_m - 2\gamma A_\gamma^{-1}, \ \ \ 
		\widetilde{D}_\gamma := I_n - 2\gamma D_\gamma^{-1}, & \\
		& M_0 := (I_{m_1} - Y_0 Z_0)^{-1}, \ \ \ N_0 := (I_{n_1} - Z_0 Y_0)^{-1}.
	\end{aligned}
\end{equation}
Note that $M_0 Y_0 = Y_0 N_0$ and $N_0 Z_0 = Z_0 M_0$. 
In terms of the matrices in \eqref{eq:sym}, we apply the SMWF~\eqref{eq:smwf} and obtain
\begin{align*}
	W_\gamma^{-1} & = A_\gamma^{-1} + U_0M_0 Y_0 V_0^{\T}, \qquad 
	&V_\gamma^{-1} &= D_\gamma^{-1} + W_0 N_0 Z_0 Q_0^{\T}, \\
	F_0 & = \widetilde{A}_\gamma - 2\gamma U_0 M_0 Y_0 V_0^{\T}, \qquad 
	&E_0 & = \widetilde{D}_\gamma - 2\gamma W_0 N_0 Z_0 Q_0^{\T}, \\
	H_0 & = 2\gamma U_0 M_0 Q_0^{\T}, \qquad 
	&G_0 &= 2\gamma W_0 N_0 V_0^{\T}. 
\end{align*}
Furthermore, let 
	$K := M_0^{-1} - 4\gamma^2 T_0 N_0 S_0$ and 
	$L := N_0^{-1} - 4\gamma^2 S_0 M_0 T_0$,
which satisfy 
\begin{equation}\label{eq:KLrelation}
	S_0 M_0 K=L N_0 S_0, \qquad    
   T_0 N_0 L = K M_0 T_0, 
\end{equation}
routine manipulations produce
\begin{align*}
	(I_m-H_0 G_0)^{-1}& = I_m + 4\gamma^2 U_0 K^{-1} T_0 N_0 V_0^{\T}, 
	& \, \,
	(I_n-G_0 H_0)^{-1}& = I_n + 4\gamma^2 W_0 L^{-1} S_0 M_0 Q_0^{\T}, 
	\\ 
	(I_m-H_0 G_0)^{-1}H_0 &= 2\gamma U_0 K^{-1} Q_0^{\T}, 
	& \, \, 
	(I_n-G_0 H_0)^{-1} G_0 &= 2\gamma W_0 L^{-1} V_0^{\T}, 
\end{align*}
where the invertibility of $K$ and $L$ comes from that of $I_m-H_0 G_0$ and $I_n-G_0 H_0$. In fact, $K$ and $L$ are nonsingular if and only if  $I_m-H_0 G_0$ and $I_n-G_0 H_0$ are nonsingular, respectively.

Since \eqref{eq:KLrelation} indicates $N_0S_0K^{-1}=L^{-1}S_0M_0$ and 
$K^{-1}T_0 N_0 = M_0T_0L^{-1}$, then with 
\begin{align*}
	&U_1 := \widetilde A_\gamma U_0, \qquad  V_1 := \widetilde A_\gamma^{\T} V_0, \qquad 
	W_1 := \widetilde D_\gamma W_0, \qquad Q_1 := \widetilde D_\gamma^{\T} Q_0, 
	\\
	&\widehat{U}_1 := [U_0, U_1], \qquad \widehat{V}_1 := [V_0, V_1], \qquad 
	\widehat{W}_1 := [W_0, W_1], \qquad \widehat{Q}_1 := [Q_0, Q_1], 
	\\
	&Y_1 := \begin{bmatrix} 0 & Y_0 \\ Y_0 & -2\gamma T_0 \end{bmatrix}, 
	\qquad 
	Z_1 := \begin{bmatrix} 0 & Z_0 \\ Z_0 & -2\gamma S_0 \end{bmatrix}, 
\end{align*}
the SDA \eqref{eq:maresda} leads to  
\begin{equation}\label{eq:F1}
	\begin{aligned}[b]
		F_1 =& \widetilde{A}_\gamma^2 - 2\gamma \widehat{U}_1 
		\begin{bmatrix}
			-2\gamma Y_0 N_0 S_0 K^{-1} Y_0 & Y_0 L^{-1} 
			\\[5pt] K^{-1}Y_0 & -2\gamma K^{-1} T_0 N_0 
		\end{bmatrix}
		\widehat{V}_1^{\T}
		\\
		=& \widetilde{A}_\gamma^2 - 2\gamma 
		\widehat{U}_1
		Y_1
		\begin{bmatrix}
			N_0 + 4\gamma^2 N_0 Z_0 T_0 L^{-1} S_0 Y_0 N_0 & -2\gamma N_0 Z_0 T_0 L^{-1} \\[5pt]
			-2\gamma L^{-1} S_0 Y_0 N_0 & L^{-1} 
		\end{bmatrix} 
		\widehat{V}_1^{\T} 
		\\
		=& \widetilde{A}_\gamma^2 - 2\gamma \widehat{U}_1 
		Y_1 \left(I_{2n_1} - Z_1 Y_1\right)^{-1} \widehat{V}_1^{\T} 
		\equiv \widetilde{A}_\gamma^2 - 2\gamma 
		\widehat{U}_1 
		\left(I_{2m_1} - Y_1 Z_1\right)^{-1} Y_1 
		\widehat{V}_1^{\T}. 
	\end{aligned}
\end{equation} 
Similarly, we have
\begin{equation}\label{eq:E1}
	\begin{aligned}[b]
		E_1 =& \widetilde{D}_\gamma^2 - 2\gamma \widehat{W}_1
		\begin{bmatrix}
			-2\gamma Z_0 M_0 T_0 L^{-1} Z_0 & Z_0 K^{-1} \\[5pt] 
			L^{-1} Z_0 & -2\gamma L^{-1} S_0 M_0 
		\end{bmatrix}
		\widehat{Q}_1^{\T}
		\\
		=& \widetilde{D}_\gamma^2 - 2\gamma \widehat{W}_1 
		Z_1
		\begin{bmatrix}
			M_0 + 4\gamma^2 M_0 Y_0 S_0 K^{-1} T_0 Z_0 M_0 & -2\gamma M_0 Y_0 S_0 K^{-1} \\[5pt] -2\gamma K^{-1} T_0 Z_0 M_0 & K^{-1} 
		\end{bmatrix} 
		\widehat{Q}_1^{\T} 
		\\
		=& \widetilde{D}_\gamma^2 - 2\gamma \widehat{W}_1 Z_1 \left(I_{2m_1} - Y_1 Z_1\right)^{-1} 
		\widehat{Q}_1^{\T}
		\equiv \widetilde{D}_\gamma^2 - 2\gamma\widehat{W}_1   \left(I_{2n_1} - Z_1 Y_1\right)^{-1} Z_1
		\widehat{Q}_1^{\T},
	\end{aligned}
\end{equation}
\begin{align}
	H_1 =&2\gamma
	\widehat{U}_1
	\begin{bmatrix}
		M_0 + 4\gamma^2 M_0 Y_0 S_0 K^{-1} T_0 N_0 Z_0  & -2\gamma M_0 Y_0   S_0 K^{-1}\\[5pt]
		-2\gamma K^{-1} T_0 N_0 Z_0  & K^{-1}
	\end{bmatrix}
	\widehat{Q}_1^{\T}\nonumber
	\\
	\equiv&2\gamma \widehat{U}_1
	\left(I_{2m_1} - Y_1 Z_1\right)^{-1} \widehat{Q}_1^{\T}, 
	\label{eq:mareH1}
	\\
	G_1 =&2\gamma 
	\widehat{W}_1
	\begin{bmatrix}
		N_0 + 4\gamma^2 N_0 Z_0 T_0 L^{-1} S_0 Y_0 N_0 & -2\gamma N_0 Z_0 T_0 L^{-1} \\[5pt] -2\gamma L^{-1} S_0 Y_0 N_0 & L^{-1} 
	\end{bmatrix}
	\widehat{V}_1^{\T} \nonumber 
	\\
	\equiv & 2\gamma \widehat{W}_1	\left(I_{2n_1} - Z_1 Y_1\right)^{-1}\widehat{V}_1^{\T}.
	\label{eq:mareG1}
\end{align}

\begin{remark}
	Checking $(I_{2n_1} - Z_1 Y_1)^{-1}$ and $(I_{2m_1} - Y_1 Z_1)^{-1}$ respectively have the forms in \eqref{eq:F1} and \eqref{eq:E1} is easy but finding the formulae as well as the recursions for $Y_j$ and $Z_j$ in the first place is nontrivial! 
\end{remark}

By pursuing  a similar process   we subsequently obtain the following theorem.

\begin{theorem}[Decoupled formulae of the dSDA for MAREs] \label{thm:MAREiterationk}
Define $U_0 := A_\gamma^{-1} B_l$, $V_0 := A_\gamma^{-{\T}} C_r$, 
$W_0 := D_\gamma^{-1} C_l$, $Q_0 := D_\gamma^{-{\T}} B_r$, 
$Y_0 := B_r^{\T} D_\gamma^{-1} C_l$, $Z_0 := C_r^{\T} A_\gamma^{-1} B_l$, 
$T_0 := Q_0^{\T} W_0$ and  $S_0 := V_0^{\T} U_0$.  
For $j \geq 1$, denote 
$U_j := \widetilde A_\gamma U_{j-1}$, $V_j := \widetilde A_\gamma^{{\T}} V_{j-1}$,  
$W_j := \widetilde D_\gamma W_{j-1}$, $Q_j :=  \widetilde D_\gamma^{{\T}} Q_{j-1}$, 
\begin{align*}
	& \widehat{U}_j := [U_0, \cdots, U_{2^j -1}], \ \  \widehat{V}_j := [V_0, \cdots, V_{2^j -1}], \ \  
	\widehat{W}_j := [W_0, \cdots, W_{2^j -1}], \ \  \widehat{Q}_j := [Q_0, \cdots, Q_{2^j -1}], 
	\\ 
	& T_j := \widehat{Q}_j^{\T} \widehat{W}_j, \qquad S_j := \widehat{V}_j^{\T} \widehat{U}_j, \qquad 
	Y_j := \begin{mat}{cc} 0 & Y_{j-1} \\ Y_{j-1} & -2\gamma T_{j-1} \end{mat}, \qquad Z_j := \begin{mat}{cc} 0 & Z_{j-1} \\ Z_{j-1} & -2\gamma S_{j-1} \end{mat}.
\end{align*} 
Assume that $I_m - H_k G_k$ and $I_n - G_k H_k$ are nonsingular for $k \geq 0$. 
For all $k\ge 1$ it holds that 
\be
	&& F_k = \widetilde{A}_\gamma^{2^k} - 2\gamma \widehat{U}_k 
	(I_{2^k m_1} - Y_k Z_k)^{-1} Y_k \widehat{V}_k^{\T},  
	\ \ \ 
	E_k = \widetilde{D}_\gamma^{2^k} - 2\gamma \widehat{W}_k (I_{2^k n_1} - Z_k Y_k)^{-1} Z_k \widehat{Q}_k^{\T},  
	\\
	&& H_k = 2\gamma \widehat{U}_k  (I_{2^k m_1} - Y_k Z_k)^{-1} \widehat{Q}_k^{\T}, 
	\qquad \qquad \quad 
	G_k = 2\gamma \widehat{W}_k (I_{2^k n_1} - Z_k Y_k)^{-1} \widehat{V}_k^{\T}.  
\ee
\end{theorem}

Again the four formulae in the SDA \eqref{eq:maresda} are decoupled. There is no reason why we need to calculate $F_k$, $E_k$ or $G_k$, if we only want to solve MAREs and control convergence using $H_k - H_{k-1}$  or the normalized residual. 

\begin{remark}\label{rk:dsda-for-adda}
	For the alternating-directional doubling algorithm  (ADDA for abbreviation) proposed in \cite{wangWL2012alternatingdirectional}, which is  a variation of the SDA, the initial items contain two parameters as follows:
	\begin{align*}
		A_\beta& := A + \beta I_m, \qquad  &  D_\alpha &:= D + \alpha I_n,  \\
		W_{\alpha,\beta} &:= A_\beta - B D_\alpha^{-1} C, \qquad &  V_{\alpha,\beta} &:= D_\alpha - C A_\beta^{-1} B, \\
		F_0& = I_m - (\beta +\alpha) W_{\alpha,\beta}^{-1}, \qquad & E_0 &= I_n - (\alpha+\beta) V_{\alpha,\beta}^{-1},  \\
		H_0 &= (\beta + \alpha) W_{\alpha,\beta}^{-1} B D_{\alpha}^{-1}, \qquad  & G_0 &= (\alpha + \beta) D_\alpha^{-1} C W_{\alpha,\beta}^{-1},  
	\end{align*}
	where $\alpha \ge \max_i a_{ii}$,  $\beta \ge \max_j d_{jj}$ with $a_{ii}$ and $d_{jj}$ respectively being the diagonal entries of $A$ and $D$. Similar to \eqref{eq:sym} we define 
	\begin{align*}
		&		Y_0 := B_r^{\T} D_\alpha^{-1} C_l, \ \ \ Z_0 := C_r^{\T} A_\beta^{-1} B_l, & \\
		& U_0 := A_\beta^{-1} B_l, \ \ V_0 := A_\beta^{-{\T}} C_r, \ \ \ 
		W_0 := D_\alpha^{-1} C_l, \ \ Q_0 := D_\alpha^{-{\T}} B_r, & \\
		& \widetilde{A}_\beta := I_m - (\alpha + \beta) A_\beta^{-1}, \ \ \ 
		\widetilde{D}_\alpha := I_n - (\alpha + \beta) D_\alpha^{-1}, & 
	\end{align*}
	then applying  the SMWF~\eqref{eq:smwf} yields  
	\be
	&& F_0  = \widetilde{A}_\beta - (\alpha + \beta) U_0 (I_{m_1} -Y_0 Z_0 )^{-1}Y_0 V_0^{\T}, \qquad 
	E_0  = \widetilde{D}_\alpha - (\alpha + \beta) W_0 (I_{n_1} - Z_0 Y_0)^{-1} Z_0 Q_0^{\T}, \\
	&& H_0 = (\alpha + \beta) U_0 (I_{m_1} -Y_0 Z_0 )^{-1} Q_0^{\T}, \qquad \qquad \quad    
	G_0 = (\alpha + \beta) W_0 (I_{n_1} - Z_0 Y_0)^{-1} V_0^{\T}. 
	\ee
Let 
		$U_1 := \widetilde A_\beta U_0$,   $V_1 := \widetilde A_\beta^{\T} V_0$,  
		$W_1 := \widetilde D_\alpha W_0$,  $Q_1 := \widetilde D_\alpha^{\T} Q_0$,
		\begin{align*}
			Y_1 := \begin{bmatrix} 0 & Y_0 \\ Y_0 & -(\alpha+\beta) T_0 \end{bmatrix}, 
			\qquad 
			Z_1 := \begin{bmatrix} 0 & Z_0 \\ Z_0 & -(\alpha+\beta) S_0 \end{bmatrix}, 
		\end{align*}
	where $T_0 := Q_0^{\T} W_0$ and	$S_0 := V_0^{\T} U_0$, and use similar notations $\widehat{U}_1$, $\widehat{V}_1$, $\widehat{W}_1$ and $\widehat{Q}_1$ as in Theorem~\ref{thm:MAREiterationk},   then by  the  manipulations analogy to \eqref{eq:F1}, \eqref{eq:E1},\eqref{eq:mareH1} and \eqref{eq:mareG1},    we get 
	\begin{align*}
		F_1&=\widetilde{A}_\beta^2 - (\alpha+\beta) 
		\widehat{U}_1
		(I_{2m_1} - Y_1 Z_1)^{-1} Y_1 \widehat{V}_1^{\T},
		\ \ \ 
		E_1 = \widetilde{D}_\alpha^2 - (\alpha+\beta) 
		\widehat{W}_1
		(I_{2n_1} - Z_1 Y_1)^{-1} Z_1
		\widehat{Q}_1^{\T}, 
		\\
		H_1 &= (\alpha+\beta) 
		\widehat{U}_1 (I_{2m_1} - Y_1 Z_1)^{-1}
		\widehat{Q}_1^{\T}, 
		\qquad \qquad  \,   
		G_1=(\alpha + \beta)
		\widehat{W}_1 (I_{2n_1} - Z_1 Y_1)^{-1} 
		\widehat{V}_1^{\T}. 
	\end{align*}
	Clearly,  with low-rank structure  $F_1, E_1, H_1$ and $G_1$ in the ADDA is decoupled. Furthermore, by   performing many similar operations we will  obtain the same  results as those in Theorem~\ref{thm:MAREiterationk}, implying that  the ADDA can be decoupled. 
\end{remark}

\begin{theorem}[Decoupled form of the ADDA for MAREs]\label{thm:decoupled-form-of-the-adda-for-mares}
	Define $U_j:=\widetilde{A}_{\beta} U_{j-1}$, $V_j:=\widetilde{A}_{\beta}^{\T} V_{j-1}$, 
	${W}_j:=\widetilde{D}_{\alpha} W_{j-1}$ and $Q_j:= \widetilde{D}_{\alpha}^{\T} Q_{j-1}$ for $j\ge 1$.  
	Assume that $I_m - H_k G_k$ and $I_n - G_k H_k$ are nonsingular for $k \geq 0$. 
	For $k\ge 2$,   denote $\widehat{Q}_k=[Q_0, Q_1, \cdots, Q_{2^k-1}]$,  $\widehat{U}_k = [U_0, U_1, \cdots, U_{2^k-1}]$, $\widehat{V}_k=[V_0, V_1, \cdots, V_{2^k-1}]$ and  $\widehat{W}_k = [W_0, W_1, \cdots, W_{2^k-1}]$, and 	
let $Y_k=\begin{bmatrix}
		0 & Y_{k-1} \\ Y_{k-1} & -(\alpha+\beta)T_{k-1}
	\end{bmatrix}$ and   $Z_k=\begin{bmatrix}
	0& Z_{k-1}\\Z_{k-1} & -(\alpha+\beta) S_{k-1}
\end{bmatrix}$  with $T_{k-1} = \widehat{Q}_{k-1}^{\T} \widehat{W}_{k-1}$ 
and $S_{k-1} = \widehat{V}_{k-1}^{\T} \widehat{U}_{k-1}$. 
The ADDA produces the following decoupled form
\begin{align*}
	& F_k = \widetilde{A}_{\beta}^{2^k} - (\alpha+\beta) \widehat{U}_k  
	(I_{2^k m_1} - Y_k Z_k)^{-1} Y_k \widehat{V}_k^{\T}, 
	\qquad 
	E_k = \widetilde{D}_{\alpha}^{2^k} - (\alpha+\beta) \widehat{W}_k
	 (I_{2^k n_1} - Z_k Y_k)^{-1} Z_k \widehat{Q}_k^{\T}, 
	 \\
	 & H_k = (\alpha + \beta) \widehat{U}_k 
	 (I_{2^k m_1} - Y_k Z_k)^{-1} \widehat{Q}_k^{\T},
	 \qquad \qquad \qquad  
	 G_k = (\alpha +\beta)\widehat{W}_k
	 (I_{2^k n_1} -Z_k Y_k)^{-1} \widehat{V}_k^{\T}. 
\end{align*}
\end{theorem}

\section{Numerical example}\label{sec:numerical-example}
In this section, we apply the  proposed dSDA  to one steel profile cooling model     
to illustrate  its feasibility  and  also   the  fault,  hence showing  the necessity  of truncation.    

\begin{example}
	We test the dSDA on one example  on  the cooling of steel rail profiles, which is available  from morWiki~\cite{community}   and whose size is $1357$.   In this example, $A\in \mathbb{R}^{1357 \times 1357}$ is  negative definite, thus stable, and  $B$ and $C^{\T}$ respectively   have $7$ and $6$ columns. To approximate the  stabilizing solution, we solve the corresponding CARE~\eqref{eq:care}.   	For stopping criteria, we use the normalized residual of the CARE: 
	\begin{align*}
	\rho ({H}_k):=&
	\frac{	\|A^{\T} H_k+ H_k A - H_k B B^{\T} H_k + C^{\T} C \|_F}
{2\|A^{\T} H_k\|_F + \|H_k B B^{\T}H_k\|_F + \|C^{\T} C\|_F }. 
	\end{align*}
We set the tolerance for $\rho(H_k)$ as $10^{-13}$ and the maximal number of iterations to  $20$. 
 
	Table~\ref{tab:residual} shows the variation  of the  normalized  residual $\rho(H_k)$ and the numerical rank of $H_k$ as determined by MATLAB (or r($H_k$)) along with the iteration index $k$. With $9$ doubling iterations the dSDA produces a stabilizing  approximation  whose relative residual is $7.614\times 10^{-15}$.  Besides,   the   computed solution  has a low rank of $191$.  The total  execution time is $60.156$ seconds  when  running   on a 64-bit  PC with an Intel Core i7 CPU at 2.70GHz and 16G RAM.  
	
For comparison, we also apply the SDA with the same parameters.  After $9$ iterates  it produces an accurate approximation of a low-rank $110$. However,  the execution time for the SDA is  only 16.194 seconds.   By comparing the   numerical results from the SDA and  the dSDA, we know that	although the proposed dSDA is feasible, it is far from satisfactory.
For instance,   the columns of   $\widehat{V}_k$ doubles in each iterate,  thus  we  compute with many insignificant and unnecessary basis vectors.  In other words, the lower right corner of $(I+Y_k Y_k^{\T})^{-1}$ attenuates rapidly when the dSDA begins to converge although its size grows doubly, so  we have to  calculate many inconsequential values. In fact, the superfluous operations can be avoid when  ``truncation'' is applied. 
\begin{table}[H]
	\centering
\begin{tabular}{c|c|c||c|c|c}
\hline
$k$ &  $\rho(H_k)$ &  $\rank(H_k)$ &  $k$ & $\rho(H_k)$ & $\rank(H_k)$    
\\
\hline 
$1$ & $3.287\times 10^{-2}$ & $12$    & $6$ & $4.513\times 10^{-10}$ & $160$
\\
$2$ & $9.694\times 10^{-4}$ & $24$    & $7$ & $1.157\times 10^{-11}$ & $178$
\\
$3$ & $2.635\times 10^{-5}$ & $48$    & $8$ & $2.969\times 10^{-13}$ & $191$
\\
$4$ & $6.852\times 10^{-7}$ & $96$    & $9$ & $7.614\times 10^{-15}$ & $191$
\\
$5$ & $1.759\times 10^{-8}$ & $139$   &     &                        & 
\\
\hline
	\end{tabular}	
	\caption{Normalized residuals and ranks}
	\label{tab:residual}
\end{table}

The example merely illustrates the validity of the dSDA. As the closely related Krylov subspace methods, it only makes sense for applications to large-scale problems, with truncation implemented (as in the dSDA$_\mathrm{t}$). 

\end{example}

\section{Conclusions}\label{sec:conclusions}
In this paper, we present a decoupled form for the classical structure-preserving doubling algorithm, the dSDA. We only need to compute with  one recursion and may apply the associated low-rank structures, solving large-scale problems efficiently. Due to the page limitation, we only present the theoretical development for the dSDA. The computation issues in practical applications, especially the truncation process to control the rank of the approximate solution, will be presented in a companion paper.

\section*{Acknowledgements}
Part of the work was completed when the first three authors visited the ST Yau Research Centre at the National Chiao Tung University, Hsinchu, Taiwan. The first author is supported in part  by NSFC-11901290 and Fundamental Research Funds for the Central Universities, and the third author is supported  in part by NSFC-11901340.

{\small
	\bibliographystyle{siam}
	\bibliography{sdals}
}

\end{document}